\documentclass[twoside,reqno,11pt]{amsart}

\usepackage{amsthm}
\usepackage{amsmath}
\usepackage{latexsym}
\usepackage{amsfonts}
\usepackage{amssymb}

\newtheoremstyle{theorem}
  {15pt}
  {15pt}
  {\sl}
  {\parindent}
  {\sc}
  {. }
  { }
  {}
\theoremstyle{theorem}

\newtheorem{theorem}{Theorem}[section]
\newtheorem{corollary}{Corollary}[section]

\newtheoremstyle{defi}
  {15pt}
  {15pt}
  {\rm}
  {\parindent}
  {\sc}
  {. }
  { }
  {}
\theoremstyle{defi}
\newtheorem{definition}{Definition}[section]
\newtheorem{remark}{Remark}[section]

\usepackage{hyperref}

\title[GREEN'S THEOREM FOR GENERALIZED FRACTIONAL  \dots]{GREEN'S THEOREM
FOR GENERALIZED\\[4pt] FRACTIONAL DERIVATIVES}

\author[T. Odzijewicz, A.B. Malinowska, D.F.M. Torres]{T. Odzijewicz $^1$, A. B. Malinowska $^2$, D. F. M. Torres $^1$}

\thanks{This is a preprint of a paper whose final and definite form will appear at 
{\tt www.springerlink.com}: Fract. Calc. Appl. Anal. {\bf 16} (2013), no.~1, in press.}

\begin{document}

\vbox to 2.5cm { \vfill }

\bigskip \medskip


\begin{abstract}
We study three types of generalized partial fractional operators.
An extension of Green's theorem, by considering partial fractional
derivatives with more general kernels, is proved.
New results are obtained, even in the particular case when the generalized
operators are reduced to the standard partial fractional derivatives
and fractional integrals in the sense of Riemann--Liouville or Caputo.

\medskip


\smallskip

\end{abstract}


\maketitle

\vspace*{-16pt}


\section{Introduction}

In 1828, the English mathematician George Green (1793-1841),
who up to his forties was working as a baker and a miller,
published an essay where he introduced a formula connecting
the line integral around a simple closed curve with a double integral.
Within years, this result turned out to be useful in many fields
of mathematics, physics and engineering \cite{Colton,Emch,Russenschuck,Sherman}.
Generalizations of Green's theorem have chosen different directions,
and are known as the Kelvin--Stokes and the Gauss--Ostrogradsky theorems.

In this paper, in contrast with previous works,
we want to state a Green's theorem for generalized
partial fractional derivatives. Notions of generalized fractional
derivatives were introduced in \cite{OmPrakashAgrawal,book:Kiryakova},
and then developed in \cite{FVC_Gen,MyID:227}.
A fractional version of the Green theorem
has been already showed for Riemann--Liouville integrals
and Caputo derivatives \cite{tarasov}, and for fractional operators
in the sense of Jumarie \cite{MyID:182}. However, generalized fractional
operators have never been considered. Our result may be useful
in the theory of fractional calculus (see, e.g.,
\cite{book:Kilbas,book:Klimek,book:Podlubny,book:Samko}),
in particular for the two-dimensional fractional calculus of variations,
where the derivation of Euler--Lagrange equations uses,
as a key step in the proof,
Green's theorem \cite{MyID:182,Cresson,book:frac,tatiana}.

The paper is organized as follows. In Section~\ref{sec:prelim} a common review
of ordinary and partial generalized fractional calculus is given. Our results
are then formulated and proved in Section~\ref{sec:MR}: we show the two-dimensional
integration by parts formula for generalized Riemann--Liouville partial
fractional integrals (Theorem~\ref{theorem:GRI})
and Green's theorem for generalized
partial fractional derivatives (Theorem~\ref{thm:ggt}).


\section{Basic Notions}
\label{sec:prelim}

In this section we give definitions of generalized ordinary
and partial fractional operators. By the choice of a certain
kernel, these operators can be reduced to the standard fractional
integrals and derivatives. For more on the subject, we refer the
reader to \cite{OmPrakashAgrawal,MR2974327,book:Kiryakova,FVC_Gen,MyID:227}.


\subsection{Generalized fractional operators}

\begin{definition}[Generalized fractional integral]
\label{def:GI}
The operator $K_P^\alpha$ is given by
\begin{equation*}
\left(K_P^{\alpha}f\right)(t)
:=p\int\limits_{a}^{t}k_{\alpha}(t,\tau)f(\tau)d\tau
+q\int\limits_{t}^{b}k_{\alpha}(\tau,t)f(\tau)d\tau,
\end{equation*}
where $P=\langle a,t,b,p,q\rangle$ is the \emph{parameter set}
($p$-set for brevity), $t\in[a,b]$, $p,q$ are real numbers,
and $k_{\alpha}(t,\tau)$ is a kernel which may depend on $\alpha$.
The operator $K_P^\alpha$ is referred as the \emph{operator $K$}
($K$-op for simplicity) of order $\alpha$ and $p$-set $P$.
\end{definition}

\begin{theorem}[Theorem 2.3 of \cite{FVC_Gen}]
\label{theorem:L1}
Let $k_\alpha$ be a difference kernel, i.e.,
$k_\alpha(t,\tau)=k_\alpha(t-\tau)$ and
$k_\alpha\in L_1\left([a,b]\right)$. Then,
$K_P^{\alpha}:L_1\left([a,b]\right)\rightarrow
L_1\left([a,b]\right)$ is well defined,
bounded and linear operator.
\end{theorem}

The $K$-op reduces to the classical left or right Riemann--Liouville
fractional integral (see, \textrm{e.g.}, \cite{book:Kilbas,book:Podlubny})
for a suitably chosen kernel $k_{\alpha}(t,\tau)$ and $p$-set $P$. Indeed,
let $k_{\alpha}(t-\tau)=\frac{1}{\Gamma(\alpha)}(t-\tau)^{\alpha-1}$.
If $P=\langle a,t,b,1,0\rangle$, then
\begin{equation*}
\left(K_{P}^{\alpha}f\right)(t)=\frac{1}{\Gamma(\alpha)}
\int\limits_a^t(t-\tau)^{\alpha-1}f(\tau)d\tau
=: \left({_{a}}\textsl{I}^{\alpha}_{t} f\right)(t)
\end{equation*}
is the left Riemann--Liouville fractional integral
of order $\alpha$; if $P=\langle a,t,b,0,1\rangle$, then
\begin{equation*}
\left(K_{P}^{\alpha}f\right)(t)=\frac{1}{\Gamma(\alpha)}
\int\limits_t^b(\tau-t)^{\alpha-1}f(\tau)d\tau
=: \left({_{t}}\textsl{I}^{\alpha}_{b} f\right)(t)
\end{equation*}
is the right Riemann--Liouville fractional integral
of order $\alpha$.

\begin{definition}[Generalized Riemann--Liouville derivative]
\label{def:GRL}
Let $P$ be a given parameter set. The operator $A_P^\alpha$,
$0 < \alpha < 1$, is defined for functions $f$ such that
$K_P^{1-\alpha} f\in AC\left([a,b]\right)$ by
$A_P^\alpha := \frac{d}{dt} \circ K_P^{1-\alpha}$,
where $D$ denotes the standard derivative.
We refer to $A_P^\alpha$ as \emph{operator $A$} ($A$-op)
of order $\alpha$ and $p$-set $P$.
\end{definition}

\begin{definition}[Generalized Caputo derivative]
\label{def:GC}
Let $P$ be a given parameter set. The operator $B_P^\alpha$,
$\alpha \in (0,1)$, is defined for functions $f$
such that $f\in AC\left([a,b]\right)$ by
$B_P^\alpha :=K_P^{1-\alpha} \circ \frac{d}{dt}$
and is referred as the \emph{operator $B$} ($B$-op)
of order $\alpha$ and $p$-set $P$.
\end{definition}

Let $k_{1-\alpha}(t-\tau)=\frac{1}{\Gamma(1-\alpha)}(t-\tau)^{-\alpha}$,
$\alpha \in (0,1)$. If $P=\langle a,t,b,1,0\rangle$, then
\begin{equation*}
\left(A_{P}^\alpha f\right)(t) = \frac{1}{\Gamma(1-\alpha)}
\frac{d}{dt} \int\limits_a^t(t-\tau)^{-\alpha}f(\tau)d\tau
=: \left({_{a}}\textsl{D}^{\alpha}_{t} f\right)(t)
\end{equation*}
is the standard left Riemann--Liouville fractional derivative
of order $\alpha$ while
\begin{equation*}
\left(B_{P}^\alpha f\right)(t)
=\frac{1}{\Gamma(1-\alpha)}
\int\limits_a^t(t-\tau)^{-\alpha} f'(\tau)d\tau
=: \left({^{C}_{a}}\textsl{D}^{\alpha}_{t} f\right)(t)
\end{equation*}
is the standard left Caputo fractional derivative of order $\alpha$;
if $P=\langle a,t,b,0,1\rangle$, then
\begin{equation*}
- \left(A_{P}^\alpha f\right)(t)
= - \frac{1}{\Gamma(1-\alpha)} \frac{d}{dt}
\int\limits_t^b(\tau-t)^{-\alpha}f(\tau)d\tau
=: \left({_{t}}\textsl{D}^{\alpha}_{b} f\right)(t)
\end{equation*}
is the standard right Riemann--Liouville
fractional derivative of order $\alpha$ while
\begin{equation*}
- \left(B_{P}^\alpha f\right)(t)
= - \frac{1}{\Gamma(1-\alpha)}
\int\limits_t^b(\tau-t)^{-\alpha} f'(\tau)d\tau
=: \left({^{C}_{t}}\textsl{D}^{\alpha}_{b} f\right)(t)
\end{equation*}
is the standard right Caputo fractional derivative of order $\alpha$.


\subsection{Generalized partial fractional operators}

Let $\alpha$ be a real number from the interval $(0,1)$,
$\Delta_n=[a_1,b_1]\times\dots\times [a_n,b_n]$, $n\in\mathbb{N}$,
be a subset of $\mathbb{R}^n$, $\textbf{t}=(t_1,\dots, t_n)$
be a point in $\Delta_n$ and $\mathbf{p}=(p_1,\dots,p_n)$,
$\mathbf{q}=(q_1,\dots,q_n)\in\mathbb{R}^n$. Generalized partial fractional
integrals and derivatives are a natural generalization
of the corresponding one-dimensional generalized fractional
integrals and derivatives.

\begin{definition}[Generalized partial fractional integral]
\label{def:GPI}
Let function $f=f(t_1,\dots,t_n)$ be continuous
on the set $\Delta_n$. The generalized partial
Riemann--Liouville fractional integral of order $\alpha$
with respect to the $i$th variable $t_i$ is given by
\begin{multline*}
\left(K_{P_{t_i}}^{\alpha}f\right)(\textbf{t})
:=p_i\int\limits_{a_i}^{t_i}k_{\alpha}(t_i,\tau)
f(t_1,\dots,t_{i-1},\tau,t_{i+1},\dots,t_n)d\tau \\
+q_i\int\limits_{t_i}^{b_i}k_{\alpha}(\tau,t_i)
f(t_1,\dots,t_{i-1},\tau,t_{i+1},\dots,t_n)d\tau,
\end{multline*}
where $P_{t_i}=\langle a_i,t_i,b_i,p_i,q_i \rangle$.
We refer to $K_{P_{t_i}}^{\alpha}$ as the
\emph{partial operator $K$} (partial $K$-op)
of order $\alpha$ and $p$-set $P_{t_i}$.
\end{definition}

\begin{definition}[Generalized partial Riemann--Liouville derivatives]
\label{def:GPRL}
Let $P_{t_i}=\langle a_i,t_i,b_i,p_i,q_i \rangle$
and $K_{P_{t_i}}^{1-\alpha}f\in C^1(\Delta_n)$.
The generalized partial Riemann--Liouville fractional
derivative of order $\alpha$ with respect
to the $i$th variable $t_i$ is given by
\begin{equation*}
\begin{split}
\bigl(A_{P_{t_i}}^{\alpha}f\bigr)(\textbf{t})
&:=\left(\frac{\partial}{\partial t_i}\circ K_{P_{t_i}}^{1-\alpha}f\right)(\textbf{t})\\
&=\frac{\partial}{\partial t_i}\left(p_i\int\limits_{a_i}^{t_i}
k_{1-\alpha}(t_i,\tau)f(t_1,\dots,t_{i-1},\tau,t_{i+1},\dots,t_n)d\tau \right.\\
&\qquad \qquad \left.+q_i\int\limits_{t_i}^{b_i}k_{1-\alpha}(\tau,t_i)
f(t_1,\dots,t_{i-1},\tau,t_{i+1},\dots,t_n)d\tau\right).
\end{split}
\end{equation*}
The operator $A_{P_{t_i}}^{\alpha}$ is referred as the
\emph{partial operator $A$} (partial $A$-op)
of order $\alpha$ and $p$-set $P_{t_i}$.
\end{definition}

\begin{definition}[Generalized partial Caputo derivative]
\label{def:GPC}
Let $P_{t_i}=\langle a_i,t_i,b_i,p_i,q_i \rangle$
and $f\in C^1(\Delta_n)$. The generalized partial Caputo
fractional derivative of order $\alpha$ with respect
to the $i$th variable $t_i$ is given by
\begin{equation*}
\begin{split}
\bigl(B_{P_{t_i}}^{\alpha}f\bigr)(\textbf{t})
&:=\left(K_{P_{t_i}}^{1-\alpha} \circ \frac{\partial}{\partial t_i} f\right)(\textbf{t})\\
&=p_i\int\limits_{a_i}^{t_i}k_{1-\alpha}(t_i,\tau)\frac{\partial}{\partial \tau}
f(t_1,\dots,t_{i-1},\tau,t_{i+1},\dots,t_n)d\tau\\
& \qquad \qquad +q_i\int\limits_{t_i}^{b_i}k_{1-\alpha}(\tau,t_i)\frac{\partial}{\partial \tau}
f(t_1,\dots,t_{i-1},\tau,t_{i+1},\dots,t_n)d\tau
\end{split}
\end{equation*}
and is referred as the \emph{partial operator $B$} (partial $B$-op)
of order $\alpha$ and $p$-set $P_{t_i}$.
\end{definition}

Similarly as in the one-dimensional case \cite{OmPrakashAgrawal,FVC_Gen,MyID:227},
the generalized partial operators $K$, $A$ and $B$ here introduced
give the standard partial fractional integrals
and derivatives for particular kernels and $p$-sets.
The left- and right-sided Riemann--Liouville partial fractional integrals
with respect to the $i$th variable $t_i$ are obtained by choosing the kernel
$$
k_{\alpha}(t_i,\tau)=\frac{1}{\Gamma(\alpha)}(t_i-\tau)^{\alpha-1}
$$
and $p$-sets $L_{t_i}=\langle a_i,t_i,b_i,1,0\rangle$
and $R_{t_i}=\langle a_i,t_i,b_i,0,1\rangle$, respectively:
\begin{equation*}
\begin{split}
\left({_{a_i}}\textsl{I}^{\alpha}_{t_i} f\right)(\textbf{t})
&= \left(K_{L_{t_i}}^{\alpha}f\right)(\textbf{t})\\
&=\frac{1}{\Gamma(\alpha)}\int\limits_{a_i}^{t_i}(t_i-\tau)^{\alpha-1}
f(t_1,\dots,t_{i-1},\tau,t_{i+1},\dots,t_n)d\tau,
\end{split}
\end{equation*}
\begin{equation*}
\begin{split}
\left({_{t_i}}\textsl{I}^{\alpha}_{b_i} f\right)(\textbf{t})
&= \left(K_{R_{t_i}}^{\alpha}f\right)(\textbf{t})\\
&=\frac{1}{\Gamma(\alpha)}\int\limits_{t_i}^{b_i}(\tau-t_i)^{\alpha-1}
f(t_1,\dots,t_{i-1},\tau,t_{i+1},\dots,t_n)d\tau.
\end{split}
\end{equation*}
The standard left- and right-sided partial Riemann--Liouville and Caputo
fractional derivatives with respect to the $i$th variable $t_i$
are obtained with the choice of kernel $k_{1-\alpha}(t_i,\tau)
=\frac{1}{\Gamma(1-\alpha)}(t_i-\tau)^{-\alpha}$:
if $P_{t_i}=\langle a_i,t_i,b_i,1,0\rangle$, then
\begin{equation*}
\begin{split}
\left({_{a_i}}\textsl{D}^{\alpha}_{t_i} f\right)(\textbf{t})
&= \left(A_{P_{t_i}}^{\alpha}f\right)(\textbf{t})\\
&=\frac{1}{\Gamma(1-\alpha)}\frac{\partial}{\partial t_i}
\int\limits_{a_i}^{t_i}(t_i-\tau)^{-\alpha}
f(t_1,\dots,t_{i-1},\tau,t_{i+1},\dots,t_n)d\tau
\end{split}
\end{equation*}
and
\begin{equation*}
\begin{split}
\left({^{C}_{a_i}}\textsl{D}^{\alpha}_{t_i} f\right)(\textbf{t})
&= \left(B_{P_{t_i}}^{\alpha}f\right)(\textbf{t})\\
&=\frac{1}{\Gamma(1-\alpha)}\int\limits_{a_i}^{t_i}
(t_i-\tau)^{-\alpha}\frac{\partial}{\partial \tau}
f(t_1,\dots,t_{i-1},\tau,t_{i+1},\dots,t_n)d\tau;
\end{split}
\end{equation*}
if $P_{t_i}=\langle a_i,t_i,b_i,0,1\rangle$, then
\begin{equation*}
\begin{split}
\left({_{t_i}}\textsl{D}^{\alpha}_{b_i} f\right)(\textbf{t})
&= -\left(A_{P_{t_i}}^{\alpha}f\right)(\textbf{t})\\
&=-\frac{1}{\Gamma(1-\alpha)}\frac{\partial}{\partial t_i}\int\limits_{t_i}^{b_i}
(\tau-t_i)^{-\alpha}f(t_1,\dots,t_{i-1},\tau,t_{i+1},\dots,t_n)d\tau
\end{split}
\end{equation*}
and
\begin{equation*}
\begin{split}
\left({^{C}_{t_i}}\textsl{D}^{\alpha}_{b_i} f\right)(\textbf{t})
&= -\left(B_{P_{t_i}}^{\alpha}f\right)(\textbf{t})\\
&=-\frac{1}{\Gamma(1-\alpha)}\int\limits_{t_i}^{b_i}(\tau-t_i)^{-\alpha}
\frac{\partial}{\partial \tau}f(t_1,\dots,t_{i-1},\tau,t_{i+1},\dots,t_n)d\tau.
\end{split}
\end{equation*}

\begin{remark}
In Definitions~\ref{def:GPI}, \ref{def:GPRL} and \ref{def:GPC},
all the variables, except $t_i$, are kept fixed. That choice
of fixed values determines a function
$f_{t_1,\dots,t_{i-1},t_{i+1},\dots,t_n}:[a_i,b_i]\rightarrow \mathbb{R}$
of one variable $t_i$:
\begin{equation*}
f_{t_1,\dots,t_{i-1},t_{i+1},\dots,t_n}(t_i)
=f(t_1,\dots,t_{i-1},t_i,t_{i+1},\dots,t_n).
\end{equation*}
By Definitions~\ref{def:GI}, \ref{def:GRL}, \ref{def:GC}
and \ref{def:GPI}, \ref{def:GPRL}, \ref{def:GPC}, we have
\begin{gather*}
\left(K_{P_{t_i}}^{\alpha}f_{t_1,\dots,t_{i-1},t_{i+1},\dots,t_n}\right)(t_i)
=\left(K_{P_{t_i}}^{\alpha}f\right)(t_1,\dots,t_{i-1},t_i,t_{i+1},\dots,t_n),\\
\left(A_{P_{t_i}}^{\alpha}f_{t_1,\dots,t_{i-1},t_{i+1},\dots,t_n}\right)(t_i)
=\left(A_{P_{t_i}}^{\alpha}f\right)(t_1,\dots,t_{i-1},t_i,t_{i+1},\dots,t_n),\\
\left(B_{P_{t_i}}^{\alpha}f_{t_1,\dots,t_{i-1},t_{i+1},\dots,t_n}\right)(t_i)
=\left(B_{P_{t_i}}^{\alpha}f\right)(t_1,\dots,t_{i-1},t_i,t_{i+1},\dots,t_n).
\end{gather*}
Therefore, as in the classical integer order case, computation of partial
generalized fractional operators is reduced to the computation
of one-variable generalized fractional operators.
\end{remark}


\section{Green's Theorem for Generalized Fractional Derivatives}
\label{sec:MR}

\begin{definition}[Dual $p$-set]
Let $P_{t_i}=\langle a_i,t_i,b_i,p_i,q_i\rangle$, $i\in\mathbb{N}$.
We denote by $P_{t_i}^{*}$ the $p$-set
$P_{t_i}^{*} = \langle  a_i,t_i,b_i,q_i,p_i\rangle$
and call it the dual of $P_{t_i}$.
\end{definition}

\begin{theorem}[Generalized 2D Integration by Parts]
\label{theorem:GRI}
Let $\alpha\in (0,1)$, $P_{t_i}=\langle a_i,t_i,b_i,p_i,q_i \rangle$
be a parameter set, and $k_\alpha$ be a difference kernel, i.e.,
$k_\alpha(t_i,\tau)=k_\alpha(t_i-\tau)$ such that
$k_\alpha\in L_1([0,b_i-a_i])$, $i=1,2$.
If $f,g,\eta_1,\eta_2\in C\left(\Delta_2\right)$, then the
generalized partial fractional integrals satisfy the following identity:
\begin{multline*}
\int\limits_{a_1}^{b_1}\int\limits_{a_2}^{b_2}
\left[g(\textbf{t})\left(K_{P_{t_1}}^{\alpha}\eta_1\right)(\textbf{t})
+f(\textbf{t})\left(K_{P_{t_2}}^{\alpha}\eta_2\right)(\textbf{t})\right]dt_2 dt_1\\
=\int\limits_{a_1}^{b_1}\int\limits_{a_2}^{b_2}
\eta_1(\textbf{t})\left[\left(K_{P_{t_1}^*}^{\alpha}g\right)(\textbf{t})\right]
+\eta_2(\textbf{t})\left[\left(K_{P_{t_2}^*}^{\alpha}f\right)(\textbf{t})\right]dt_2 dt_1,
\end{multline*}
where $P_{t_i}^*$ is the dual of $P_{t_i}$, $i=1,2$.
\end{theorem}

\begin{proof}
Define
\[
F_1(\textbf{t},\tau):=
\left\{
\begin{array}{ll}
\left|p_1 k_\alpha(t_1-\tau)\right|
\cdot\left|g(\textbf{t})\right|\cdot \left|\eta_1(\tau,t_2)\right|
& \mbox{if $\tau \leq t_1$}\\
\left|q_1 k_\alpha(\tau-t_1)\right|
\cdot \left|g(\textbf{t})\right|\cdot\left|\eta_1(\tau,t_2)\right|
& \mbox{if $\tau > t_1$}
\end{array}\right.
\]
for all $(\textbf{t},\tau)\in [a_1,b_1]\times [a_2,b_2]\times [a_1,b_1]$ and
\[
F_2(\textbf{t},\tau):=
\left\{
\begin{array}{ll}
\left|p_2 k_\alpha(t_2-\tau)\right|
\cdot\left|f(\textbf{t})\right|\cdot \left|\eta_2(t_1,\tau)\right|
& \mbox{if $\tau \leq t_2$}\\
\left|q_2 k_\alpha(\tau-t_2)\right|
\cdot \left|f(\textbf{t})\right|\cdot\left|\eta_2(t_1,\tau)\right|
& \mbox{if $\tau > t_2$}
\end{array}\right.
\]
for all $(\textbf{t},\tau)\in [a_1,b_1]\times [a_2,b_2]\times [a_2,b_2]$.
Since $f,g$ and $\eta_i$, $i=1,2$, are continuous functions on $\Delta_2$,
they are bounded on $\Delta_2$. Hence, there exist real numbers
$C_1,C_2,C_3,C_4>0$ such that
$$
\left|f(\textbf{t})\right|\leq C_1, \quad
\left|g(\textbf{t})\right| \leq C_2, \quad
\left|\eta_1(\textbf{t})\right| \leq C_3,
\quad \left|\eta_2(\textbf{t})\right| \leq C_4,
$$
for all $\textbf{t}\in \Delta_2$. Therefore,
\begin{equation*}
\begin{split}
\int_{a_1}^{b_1}&\int_{a_2}^{b_2}\int_{a_1}^{b_1}
F_1(\textbf{t},\tau)dt_1 dt_2 d\tau
+ \int_{a_1}^{b_1} \int_{a_2}^{b_2} \int_{a_2}^{b_2}
F_2(\textbf{t},\tau)dt_2 d\tau dt_1\\
&=\int_{a_1}^{b_1}\left(\int_{a_2}^{b_2}\left(
\int_{\tau}^{b_1} \left|p_1 k_\alpha(t_1-\tau)\right|
\cdot\left|g(\textbf{t})\right|\cdot \left|\eta_1(\tau,t_2)\right| dt_1\right.\right. \\
&\quad + \left.\left.\int_{a_1}^{\tau} \left|q_1 k_\alpha(\tau-t_1)\right|
\cdot \left|g(\textbf{t})\right|\cdot\left|\eta_1(\tau,t_2)\right| dt_1 \right)dt_2 \right) d\tau \\
&\quad +\int_{a_1}^{b_1}\left(\int_{a_2}^{b_2}\left(\int_{\tau}^{b_2} \left|p_2 k_\alpha(t_2-\tau)\right|
\cdot\left|f(\textbf{t})\right|\cdot \left|\eta_2(t_1,\tau)\right| dt_2 \right.\right.\\
&\quad +\left.\left.\int_{a_2}^{\tau} \left|q_2 k_\alpha(\tau-t_2)\right|
\cdot \left|f(\textbf{t})\right|\cdot\left|\eta_2(t_1,\tau)\right| dt_2 \right)d\tau \right)dt_1\\
&\leq C_2 C_3 \left[\int_{a_1}^{b_1}\left(\int_{a_2}^{b_2}\left(\int_{\tau}^{b_1}
\left|p_1 k_\alpha(t_1-\tau)\right|dt_1\right.\right.\right. \\
&\quad +\left.\left.\left.\int_{a_1}^{\tau}
\left|q_1 k_\alpha(\tau-t_1)\right| dt_1\right)dt_2\right)d\tau\right]\\
&\quad + C_1 C_4 \left[\int_{a_1}^{b_1}\left(\int_{a_2}^{b_2}
\left(\int_{\tau}^{b_2}\left|p_2 k_\alpha(t_2-\tau)\right| dt_2\right.\right.\right.\\
&\quad + \left.\left.\left. \int_{\tau}^{b_2}\left|q_2 k_\alpha(\tau-t_2)\right|dt_2
\right)d\tau \right)dt_1\right]\\
&\leq C_2 C_3 \left[\int_{a_1}^{b_1}\left(\int_{a_2}^{b_2}\left(\int_{0}^{b_1-a_1}
\left|p_1 k_\alpha(u_1)\right|du_1 \right.\right.\right.\\
&\quad + \left.\left.\left.\int_{0}^{b_1-a_1}
\left|q_1 k_\alpha(u_1)\right| du_1\right)dt_2\right)d\tau\right]\\
&\quad + C_1 C_4 \left[\int_{a_1}^{b_1}\left(\int_{a_2}^{b_2}\left(
\int_{0}^{b_2-a_2}\left|p_2 k_\alpha(u_2)\right| du_2\right.\right.\right.\\
&\quad + \left.\left.\left.\int_{0}^{b_2-a_2}
\left|q_2 k_\alpha(u_2)\right|du_2 \right)d\tau \right)dt_1\right]\\
&= C_2 C_3 \left(\left|p_1\right|+\left|q_1\right|\right)\left\|k_{\alpha}\right\|(b_2-a_2)(b_1-a_1)\\
&\quad + C_1 C_4 \left(\left|p_2\right|+\left|q_2\right|\right)
\left\|k_{\alpha}\right\|(b_2-a_2)(b_1-a_1)\\
&< \infty.
\end{split}
\end{equation*}
Hence, we can use Fubini's theorem to change
the order of integration in the iterated integrals:
\begin{equation*}
\begin{split}
\int_{a_1}^{b_1}&\int_{a_2}^{b_2}\left[g(\textbf{t})\left(K_{P_{t_1}}^{\alpha}\eta_1\right)(\textbf{t})
+f(\textbf{t})\left(K_{P_{t_2}}^{\alpha}\eta_2\right)(\textbf{t})\right]dt_2 dt_1
\end{split}
\end{equation*}
\begin{equation*}
\begin{split}
&= \int_{a_1}^{b_1}\int_{a_2}^{b_2}\left[g(\textbf{t})\left(p_1 \int_{a_1}^{t_1}
k_\alpha (t_1-\tau)\eta_1(\tau,t_2)d\tau\right.\right.\\
&\quad + \left. q_1\int_{t_1}^{b_1} k_\alpha (\tau-t_1)\eta_1(\tau,t_2)d\tau\right)\\
&\quad +f(\textbf{t})\left(p_2 \int_{a_2}^{t_2} k_\alpha (t_2-\tau)\eta_2(t_1,\tau)d\tau\right.\\
&\quad + \left.\left. q_2\int_{t_2}^{b_2} k_\alpha (\tau-t_2)\eta_2(t_1,\tau)d\tau\right)\right]dt_2 dt_1\\
&=\int_{a_2}^{b_2}\left(\int_{a_1}^{b_1}\eta_1(\tau,t_2)\left(p_1\int_{\tau}^{b_1}
k_\alpha (t_1-\tau)g(\textbf{t})dt_1\right.\right.\\
&\quad + \left.\left. q_1\int_{a_1}^{\tau}k_\alpha (\tau-t_1)g(\textbf{t})dt_1\right)d\tau \right)dt_2\\
& \quad +\int_{a_1}^{b_1}\left(\int_{a_2}^{b_2}\eta_2(t_1,\tau)\left(p_2 \int_{\tau}^{b_2}
k_\alpha(t_2-\tau)f(\textbf{t})dt_2\right.\right.\\
&\quad + \left.\left. q_2 \int_{a_2}^{\tau}k_\alpha(\tau-t_2)f(\textbf{t})dt_2\right)d\tau\right)dt_1\\
&=\int_{a_1}^{b_1}\int_{a_2}^{b_2}\eta_1(\tau,t_2)\left(K_{P_{t_1^*}}^{\alpha}g\right)(\tau,t_2) dt_2 d\tau \\
&\quad + \int_{a_1}^{b_1}\int_{a_2}^{b_2}\eta_2(t_1,\tau)\left(K_{P_{t_2^*}}^{\alpha}f\right)(t_1,\tau) d\tau dt_1.
\end{split}
\end{equation*}
\end{proof}

We are now in conditions to state and prove the main result of the paper:
the Green theorem for generalized fractional derivatives.

\begin{theorem}[Generalized Green's Theorem]
\label{thm:ggt}
Let $0<\alpha<1$ and $f,g,\eta\in C^1\left(\Delta_2\right)$.
Let $k_\alpha$ be a difference kernel, i.e.,
$k_\alpha(t_i,\tau)=k_\alpha(t_i-\tau)$ such that $k_\alpha\in L_1([0,b_i-a_i])$,
$i=1,2$, and $K_{P_{t_1}^*}^{1-\alpha}g,K_{P_{t_2}^*}^{1-\alpha}f\in C^1\left(\Delta_2\right)$.
Then, the following formula holds:
\begin{multline*}
\int_{a_1}^{b_1}\int_{a_2}^{b_2}\left[g(\textbf{t})\left(B_{P_{t_1}}^{\alpha}\eta\right)(\textbf{t})
+f(\textbf{t})\left(B_{P_{t_2}}^{\alpha}\eta\right)(\textbf{t})\right]dt_2 dt_1\\
=-\int_{a_1}^{b_1}\int_{a_2}^{b_2}\eta(\textbf{t})\left[\left(A_{P_{t_1}^*}^{\alpha}g\right)(\textbf{t})
+\left(A_{P_{t_2}^*}^{\alpha}f\right)(\textbf{t})\right]dt_2 dt_1\\
+\oint_{\partial\Delta_2}\eta(\textbf{t})\left[\left(K_{P_{t_1}^*}^{1-\alpha}g\right)(\textbf{t})dt_2
-\left(K_{P_{t_2}^*}^{1-\alpha}f\right)(\textbf{t})dt_1\right].
\end{multline*}
\end{theorem}

\begin{proof}
By the definition of generalized partial Caputo fractional derivative,
Theorem~\ref{theorem:GRI}, and the standard Green's theorem, one has
\begin{equation*}
\begin{split}
\int_{a_1}^{b_1}&\int_{a_2}^{b_2}\left[g(\textbf{t})\left(B_{P_{t_1}}^{\alpha}\eta\right)(\textbf{t})
+f(\textbf{t})\left(B_{P_{t_2}}^{\alpha}\eta\right)(\textbf{t})\right]dt_2 dt_1\\
&=\int_{a_1}^{b_1}\int_{a_2}^{b_2}\left[g(\textbf{t})\left(
K_{P_{t_1}}^{1-\alpha}\frac{\partial}{\partial t_1} \eta\right)(\textbf{t})
+ f(\textbf{t})\left(K_{P_{t_2}}^{1-\alpha}\frac{\partial}{\partial t_2}
\eta\right)(\textbf{t})\right]dt_2 dt_1\\
&=\int_{a_1}^{b_1}\int_{a_2}^{b_2}\left[\frac{\partial}{\partial t_1}\eta(\textbf{t})
\left(K_{P_{t_1}^*}^{1-\alpha} g\right)(\textbf{t})
+ \frac{\partial}{\partial t_2}\eta(\textbf{t})
\left(K_{P_{t_2}^*}^{1-\alpha} f\right)(\textbf{t})\right]dt_2 dt_1\\
&=-\int_{a_1}^{b_1}\int_{a_2}^{b_2}\eta(\textbf{t})\left[\frac{\partial}{\partial t_1}
\left(K_{P_{t_1^*}}^{1-\alpha}g\right)(\textbf{t})
+ \frac{\partial}{\partial t_2}\left(
K_{P_{t_2^*}}^{1-\alpha}f\right)(\textbf{t})\right]dt_2 dt_1\\
&\quad + \oint_{\partial\Delta_2}\eta(\textbf{t})\left[\left(K_{P_{t_1}^*}^{1-\alpha}g\right)(\textbf{t})dt_2
-\left(K_{P_{t_2}^*}^{1-\alpha}f\right)(\textbf{t})dt_1\right].
\end{split}
\end{equation*}
\end{proof}

\begin{corollary}
Let $0<\alpha<1$ and $f,g,\eta\in C^1\left(\Delta_2\right)$.
If $\left({_{t_1}}\textsl{I}^{1-\alpha}_{b_1} g\right)(\textbf{t})$
and $\left({_{t_2}}\textsl{I}^{1-\alpha}_{b_2} f\right)(\textbf{t})$
are continuously differentiable on the rectangle $\Delta_2$, then
\begin{multline*}
\int_{a_1}^{b_1}\int_{a_2}^{b_2}\left[g(\textbf{t})\left({^{C}_{a_1}}\textsl{D}^{\alpha}_{t_1}
\eta\right)(\textbf{t})+f(\textbf{t})\left({^{C}_{a_2}}\textsl{D}^{\alpha}_{t_2}
\eta\right)(\textbf{t})\right]dt_2 dt_1\\
=\int_{a_1}^{b_1}\int_{a_2}^{b_2}\eta(\textbf{t})\left[
\left({_{t_1}}\textsl{D}^{\alpha}_{b_1} g\right)(\textbf{t})
+\left({_{t_2}}\textsl{D}^{\alpha}_{b_2} f\right)(\textbf{t})\right]dt_2 dt_1\\
+\oint_{\partial\Delta_2}\eta(\textbf{t})\left[\left({_{t_1}}\textsl{I}^{1-\alpha}_{b_1}
g\right)(\textbf{t})dt_2-\left({_{t_2}}\textsl{I}^{1-\alpha}_{b_2} f\right)(\textbf{t})dt_1\right].
\end{multline*}
\end{corollary}


\section*{Acknowledgements}

This work received \emph{The Grunwald--Letnikov Award: Best Student Paper (theory)},
at the 2012 Symposium on Fractional Differentiation and Its Applications (FDA'2012),
May 16, 2012, Hohai University, Nanjing.
It was supported by {\it FEDER} funds through
{\it COMPETE} (Operational Program Factors of Competitiveness)
and by Portuguese funds through the
{\it Center for Research and Development
in Mathematics and Applications} (CIDMA)
and the Portuguese Foundation for Science and Technology (FCT),
within project PEst-C/MAT/UI4106/2011
with COMPETE number FCOMP-01-0124-FEDER-022690.
Odzijewicz was also supported by FCT under Ph.D. fellowship
SFRH/BD/33865/2009; Malinowska by Bia{\l}ystok
University of Technology grant S/WI/02/2011;
and Torres by FCT through the project PTDC/MAT/113470/2009.



\bigskip \smallskip

\it


\noindent
$^1$ Center for Research and Development in Mathematics and Applications \\
Department of Mathematics \\
University of Aveiro \\
3810-193 Aveiro, PORTUGAL \\[4pt]
e-mail: tatianao@ua.pt, delfim@ua.pt
\hfill Received: July 09, 2012 \\[12pt]
$^2$ Faculty of Computer Science \\
Bia{\l}ystok University of Technology \\
15-351 Bia\l ystok, POLAND \\[4pt]
e-mail: a.malinowska@pb.edu.pl


\end{document}